\documentclass[letterpaper, 10 pt, conference]{ieeeconf}  % Comment this line out if you need a4paper

\IEEEoverridecommandlockouts   
\usepackage{booktabs} % For formal tables
\usepackage{graphics} % for pdf, bitmapped graphics files
\usepackage{epsfig} % for postscript graphics files
\usepackage{epstopdf}
\usepackage{array}
\usepackage{amsmath} % assumes amsmath package installed
\usepackage{amssymb}  % assumes amsmath package installed
\usepackage{bm}
\usepackage{stmaryrd}
\usepackage{color}
\usepackage{times} 
\usepackage{graphicx}
\usepackage{cases}
\usepackage{subfig}
\usepackage{url}
\usepackage{cite}

\newtheorem{Assumption 1}{Assumption}
\newtheorem{Remark 1}{Remark}
\newtheorem{Lemma 1}{Lemma}
\newtheorem{Corollary 1}{Corollary}
\newtheorem{theorem}{Theorem}

\title{\LARGE \bf
Privacy-preserving Decentralized Optimization via Decomposition }

\author{Chunlei Zhang and Huan Gao and Yongqiang Wang% <-this % stops a space
	\thanks{Chunlei Zhang, Huan Gao, and Yongqiang Wang are with the department of Electrical and Computer Engineering, Clemson University, Clemson, SC 29634, USA {\tt\small \{chunlez, hgao2, yongqiw\}@clemson.edu}}%
}

\begin{document}
	\maketitle
	\thispagestyle{empty}
	\pagestyle{empty}
\begin{abstract}
	This paper considers the problem of privacy-preservation in decentralized optimization, in which $N$ agents cooperatively minimize a global objective function that is the sum of  $N$ local objective functions. We assume that each local objective function is private and only known to an individual agent. To cooperatively solve the problem, most existing decentralized optimization approaches require  participating agents to exchange and disclose estimates to neighboring agents. However, this results in leakage of private information about local objective functions, which is undesirable when adversaries exist and try to steal information from participating agents. To address this issue, we propose a privacy-preserving decentralized optimization approach based on proximal Jacobian ADMM via function decomposition.   Numerical simulations confirm the effectiveness of the  proposed approach. 
\end{abstract}			
\section{Introduction}
	We consider the problem of privacy-preservation in decentralized optimization  where $N$ agents cooperatively minimize a global objective function of the following form: 
	\begin{equation}\label{eq:global function}
		\begin{aligned}
		\mathop {\min }\limits_{\tilde{\bm{x}}} \qquad \bar{f}(\tilde{\bm{x}})=\sum\limits_{i=1}^{N}f_i(\tilde{\bm{x}}), 
		\end{aligned}
	\end{equation}
	where variable $\tilde{\bm{x}}\in \mathbb{R}^n$ is common to all agents, function $f_i: \mathbb{R}^n\to\mathbb{R}$ is a private local objective function of agent $i$. 
	This problem has found  wide applications in various domains,  ranging from  rendezvous in multi-agent systems \cite{lin2004multi}, support vector machine  \cite{cortes1995support}  and classification \cite{zhang2017dynamic} in machine learning, source localization in sensor networks \cite{zhang2018distributed}, to data regression in statistics \cite{mateos2010distributed,liu2009large}. 

	To solve the optimization problem \eqref{eq:global function} in an decentralized manner, different algorithms were proposed in recent years, including the distributed (sub)gradient algorithm \cite{nedic2009distributed}, augmented Lagrangian methods (ALM) \cite{he2016proximal}, and the alternating direction method of multipliers (ADMM) as well as its variants \cite{boyd2011distributed,he2016proximal,ling2013decentralized,ling2016weighted}.  Among existing approaches, ADMM has attracted tremendous attention due to its wide applications \cite{boyd2011distributed} and fast convergence rate in both primal and dual iterations \cite{ling2016weighted}. ADMM yields a convergence rate of $O(1/k)$ when the local objective functions $f_i$ are convex and a Q-linear convergence rate when all the local objective functions are strongly convex \cite{shi2014linear}.  In addition, a recent work shows that ADMM can achieve a Q-linear convergence rate even when the local objective functions are only convex (subject to the constraint that the  global objective function is strongly convex) \cite{maros2018q}.
	
	On the other hand, privacy has become one of the key concerns. For example, in source localization, participating agents may want to reach consensus on the source position without revealing their position information \cite{al2017proloc}. In the rendezvous problem, a group of individuals may want to meet at an agreed time and place \cite{lin2004multi} without leaking their initial locations  \cite{mo2017privacy}. In the business sector, independent companies may want to work together to complete a common business for mutual benefit but without sharing their private data \cite{weeraddana2012per}. In the agreement problem \cite{degroot1974reaching}, a group of organizations may want to reach consensus on a subject without leaking their individual opinions to others  \cite{mo2017privacy}. 
	
	One widely used approach to enabling privacy-preservation in decentralized optimization is differential privacy \cite{huang2015differentially,han2017differentially,nozari2016differentially,zhang2018improving} which protects sensitive information by adding carefully-designed noise to exchanged states or objective functions. However, adding noise also compromises the accuracy of optimization results and causes a fundamental trade-off between privacy and accuracy \cite{huang2015differentially,han2017differentially,nozari2016differentially}. In fact, approaches based on differential privacy may fail to converge to the accurate optimization result even without noise perturbation \cite{nozari2016differentially}. It is worth noting that there exists some differential-privacy based optimization approaches  which are able to converge  to the accurate optimization result in the mean-square sense, e.g.  \cite{hale2015differentially,hale2017cloud}. However, those results require the assistance of a third party such as a cloud  \cite{hale2015differentially,hale2017cloud}, and therefore cannot be applied to the completely decentralized setting  where no third parties exist. Encryption-based approaches are also commonly used to enable privacy-preservation \cite{xu2015secure,yao1982protocols,zhang2018admm}. However, such approaches unavoidably bring about extra computational and communication burden for real-time optimization  \cite{lagendijk2013encrypted}. Another approach to enabling privacy preservation in linear multi-agent networks  is observability-based design \cite{pequito2014design,alaeddini2017adaptive}, which protects agents' information from non-neighboring agents through properly designing the weights of the communication graph. However, this approach cannot protect the privacy of adversary's direct neighbors. 
	
	To enable privacy in decentralized optimization without incurring large communication/computational overhead or compromising algorithmic accuracy, we propose a novel privacy solution through function decomposition. 
	In the optimization literature, privacy has been defined as  preserving the confidentiality of agents' states  \cite{hale2015differentially},  (sub)gradients or objective functions \cite{nozari2016differentially,yan2013distributed,lou2017privacy}.  In this paper, we define privacy as the non-disclosure of agents' (sub)gradients.  
	We protect agents' (sub)gradients because in many decentralized optimization applications, subgradients contain sensitive information such as salary or medical record \cite{zhang2018admm,yan2013distributed}. 

	\emph{Contributions:} We proposed a privacy-preserving decentralized optimization approach through function decomposition.  In contrast to differential-privacy based approaches which use noise to cover sensitive information and are subject to a fundamental  trade-off between privacy and accuracy, our approach can enable privacy preservation without sacrificing accuracy. Compared with encryption-based approaches which suffer from heavy computational and communication burden, our approach incurs little extra computational  and communication overhead.

	\emph{Organization:} The rest of this paper is organized as follows: Sec. \uppercase\expandafter{\romannumeral2} introduces the attack model and presents the proximal Jacobian ADMM solution to \eqref{eq:global function}. Then  a completely decentralized   privacy-preserving approach to problem \eqref{eq:global function} is proposed in Sec. \uppercase\expandafter{\romannumeral3}. Rigorous analysis of the guaranteed privacy and convergence is addressed in Sec. \uppercase\expandafter{\romannumeral4} and Sec. \uppercase\expandafter{\romannumeral5}, respectively.  Numerical simulation results are provided in Sec. \uppercase\expandafter{\romannumeral6} to confirm the effectiveness and computational efficiency of the proposed approach. In the end, we draw conclusions in Sec. \uppercase\expandafter{\romannumeral7}.

\section{Background}
	We first introduce the attack model considered in this paper.  Then we  present the proximal Jacobian  ADMM algorithm for decentralized optimization. 

	\subsection{Attack Model}
		We consider two types of adversaries in this paper:	{\it Honest-but-curious adversaries} and  {\it External eavesdroppers}. {\it Honest-but-curious adversaries}  are agents who follow all protocol steps correctly but are curious and collect all intermediate and input/output data in an attempt to learn some information about other participating agents \cite{li2010secure}. {\it External eavesdroppers} are adversaries who steal information through wiretapping all communication channels and intercepting messages exchanged between agents.

	\subsection{Proximal Jacobian ADMM}
		The decentralized problem \eqref{eq:global function} can be formulated as follows: each $f_i$ in \eqref{eq:global function} is private and only known to agent $i$, and all $N$ agents form a bidirectional connected network, which is denoted by a graph $G=(V,E)$. $V$ denotes the set of agents, $E$ denotes the set of communication links (undirected edges) between agents, and $|E|$ denotes the number of communication links (undirected edges) in $E$.  If there exists a communication link between agents $i$ and $j$, we say that agent $i$ and agent $j$ are neighbors and the  link is denoted as
		$e_{i,j}\in E$ if $i<j$ is true or $e_{j,i}\in E$ otherwise.  Moreover, the set of all neighboring agents of $i$ is denoted as $\mathcal{N}_i$ and the number of  agents in $\mathcal{N}_i$ is denoted as $D_i$.  Then problem \eqref{eq:global function} can be rewritten as 		
		\begin{equation}\label{eq:conventional admm form}
			\begin{aligned}
				&\mathop {\min }\limits_{\bm{x}_{i}\in \mathbb{R}^n,\,i\in\{1,2,\ldots,N\}} \qquad \sum\limits_{i=1}^{N}f_i(\bm{x}_{i})  \\
				&\textrm{subject to} \qquad \bm{x}_{i}=\bm{x}_{j}, \quad \forall e_{i,j}\in E,  
			\end{aligned}
		\end{equation}
		where $\bm{x}_i$ is a copy of $\tilde{\bm{x}}$ and belongs to agent $i$. Using the following proximal Jacobian ADMM \cite{deng2017parallel}, an optimal solution to \eqref{eq:global function} can be achieved at each agent:
		
		\begin{numcases}{}
			\bm{x}_i^{k+1}=\underset{\bm{x}_i}{\operatorname{argmin}} f_i(\bm{x}_i) +\frac{\gamma_i\rho}{2}\parallel \bm{x}_i-\bm{x}_i^k \parallel^2\nonumber \\
			\qquad\qquad+\sum\limits_{j\in \mathcal{N}_i}(\bm{\lambda}_{i,j}^{kT}(\bm{x}_i-\bm{x}_{j}^k)+\frac{\rho}{2}\parallel\bm{x}_i-\bm{x}_{j}^k \parallel ^2), \label{eq:conventional admm 1}\\
			\bm{\lambda}_{i,j}^{k+1}=\bm{\lambda}_{i,j}^k+\rho(\bm{x}_i^{k+1}-\bm{x}_j^{k+1}), \quad \forall  j\in\mathcal{N}_i. \label{eq:conventional admm 2}
		\end{numcases}
		Here, $k$ is the iteration index,  $\gamma_i>0$ $(i=1,2,\ldots,N)$ are proximal coefficients, $\rho$ is the penalty parameter, which is a positive constant scalar. $\bm{\lambda}_{i,j}$ and $\bm{\lambda}_{j,i}$ are Lagrange multipliers corresponding to the constraint $\bm{x}_{i}=\bm{x}_{j}, e_{i,j}\in E$.  It is worth noting that similar to our prior work \cite{zhang2018admm}, both $\bm{\lambda}_{i,j}$ and $\bm{\lambda}_{j,i}$ are introduced for the constraint $\bm{x}_{i}=\bm{x}_{j}, e_{i,j}\in E$ in \eqref{eq:conventional admm 1}-\eqref{eq:conventional admm 2} to unify the algorithm description.  By setting $\bm{\lambda}_{i,j}^{0}=\rho(\bm{x}_i^{0}-\bm{x}_j^{0})$ at $t=0$, we have $\bm{\lambda}_{i,j}^k=-\bm{\lambda}_{j,i}^k$ holding for all $i=1,2,\cdots, N, j\in\mathcal{N}_i$. In this way, the update rule of agent $i$ can be unified without separating $i>j$ and $i<j$ for $j\in\mathcal{N}_i$, as shown in \eqref{eq:conventional admm 1}.
			
		The proximal Jacobian  ADMM is effective  in solving \eqref{eq:global function}. However, it cannot protect the privacy of participating agents' gradients as states $\bm{x}_i^k$ are exchanged and disclosed explicitly among neighboring agents. So adversaries can easily derive $\triangledown f_i(\bm{x}_{i}^{k})$  explicitly for $k=1,2,\ldots$ according to the update rules in \eqref{eq:conventional admm 1} and \eqref{eq:conventional admm 2} by leveraging the knowledge of $\gamma_i$.

\section{Privacy-preserving Decentralized Optimization}
	The key idea of our approach to enabling privacy-preservation is to randomly decompose each $f_i$ into two parts $f_i^{\alpha k}$ and $f_i^{\beta k}$ under the constraint $f_i =f_i^{\alpha k}+ f_i^{\beta k}$. The index $k$ of functions $f_i^{\alpha k}$ and $f_i^{\beta k}$   indicates that functions $f_i^{\alpha k}$ and $f_i^{\beta k}$  can be time-varying. However, it should be noticed  that the sum of $f_i^{\alpha k}$ and $f_i^{\beta k}$ is time invariant and always equals to $f_i$.  We let the function  $f_i^{\alpha k}$  succeed the role of the original function $f_i$ in inter-agent interactions while the other function $f_i^{\beta k}$  involves only by interacting with  $f_i^{\alpha k}$, as shown in Fig. \ref{fig:G}. 
		
	\begin{figure}
		\begin{center}
			\subfloat[]{\includegraphics[width=0.21\textwidth]{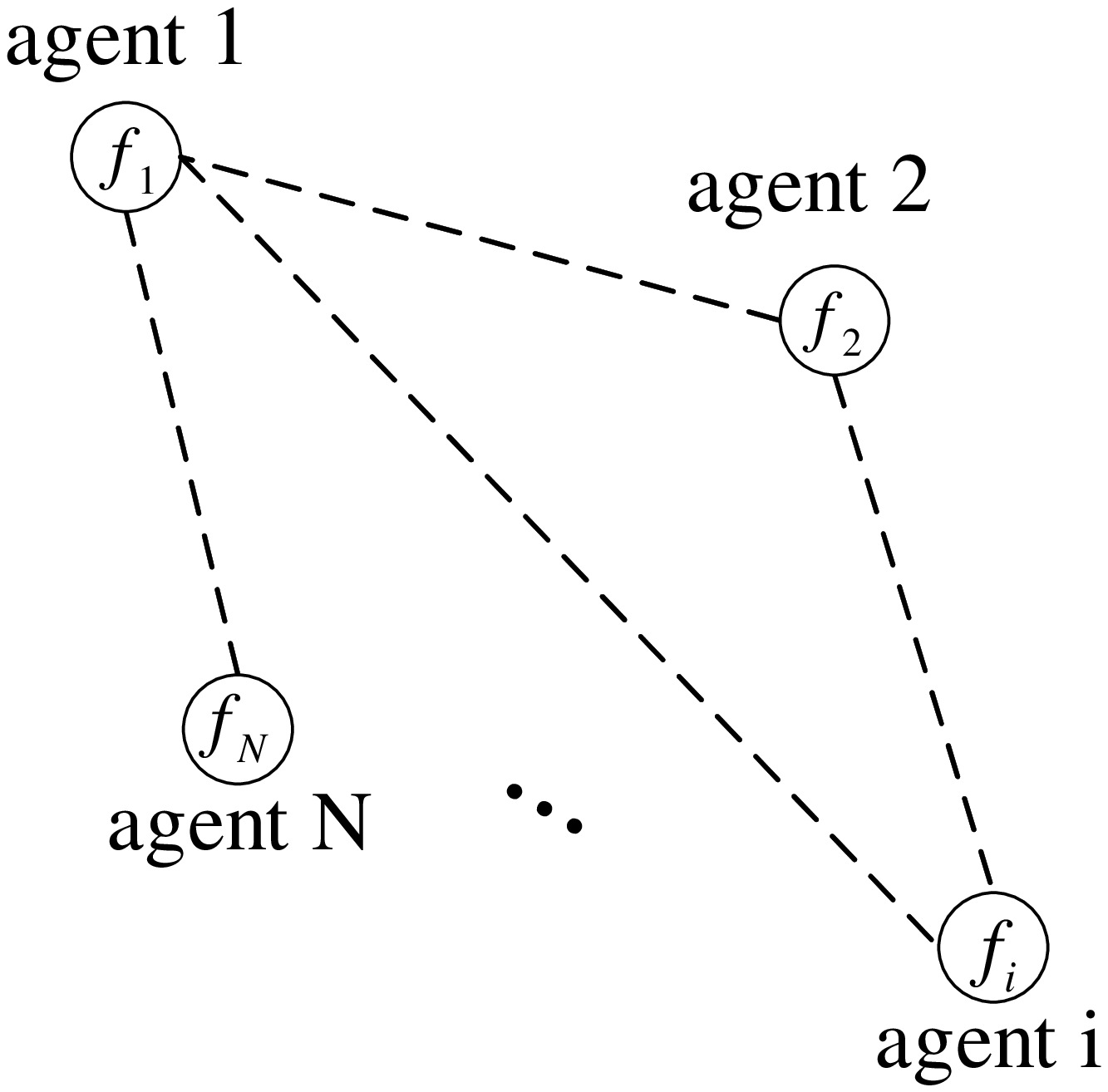}%
				\label{fig: before decomposition}}
			\hspace{.3in} 
			\subfloat[]{\includegraphics[width=0.22\textwidth]{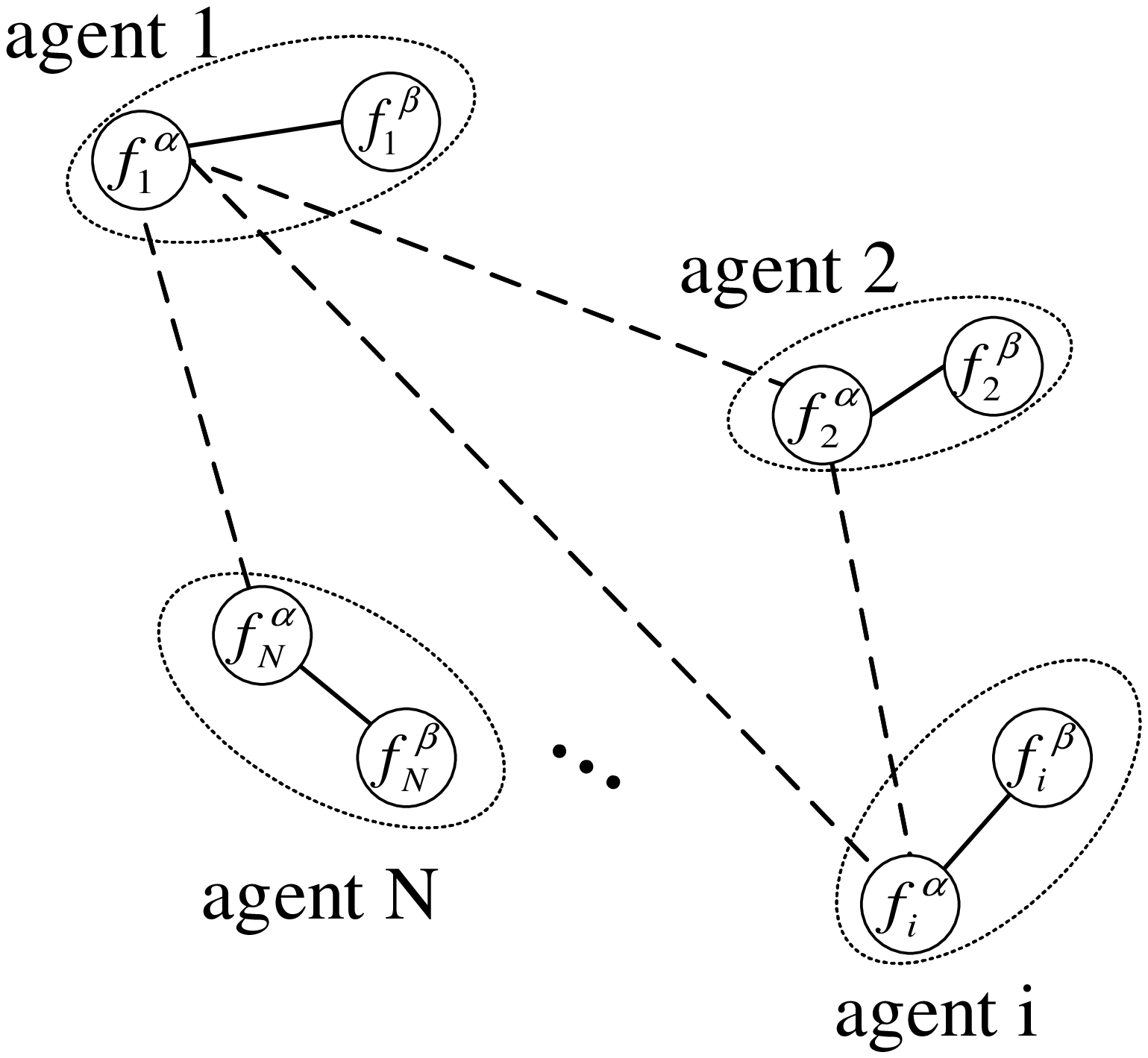}%
				\label{fig: after decomposition}}
		\end{center}
		\caption{Function-decomposition based privacy-preserving decentralized optimization. (a) Before function decomposition (b) after function decomposition.}
		\label{fig:G}
	\end{figure}
	
	After the function decomposition, problem \eqref{eq:global function} can be rewritten as 
	
	\begin{equation}\label{eq:decomposition admm form}
		\begin{aligned}
			&\mathop {\min }\limits_{\bm{x}_{i}^{\alpha},\bm{x}_{i}^{\beta} \in\mathbb{R}^n,\,i\in\{1,2,\ldots,N\}} \qquad \sum\limits_{i=1}^{N}(f_i^{\alpha k}(\bm{x}_{i}^{\alpha})+f_i^{\beta k}(\bm{x}_{i}^{\beta})) \\
			&\qquad\quad\textrm{subject to} \qquad\quad \bm{x}_{i}^{\alpha}=\bm{x}_{j}^{\alpha}, \quad \forall e_{i,j}\in E,  \\
			&\qquad\qquad\qquad\qquad\qquad \bm{x}_{i}^{\alpha}=\bm{x}_{i}^{\beta}, \quad \forall i\in V, 
		\end{aligned}
	\end{equation}
	and the associated augmented Lagrangian function is
	\begin{equation} \label{eq:decompositionaugemented lagrangian function}
		\begin{aligned}
			\lefteqn{\mathcal{L}_\rho^k(\bm{x},\bm{\lambda})=\sum\limits_{i=1}^{N}(f_i^{\alpha k}(\bm{x}_{i}^{\alpha}) +f_i^{\beta k}(\bm{x}_{i}^{\beta}) )}\\
			& \qquad+ \sum\limits_{e_{i,j}\in E}(\bm{\lambda}_{i,j}^{\alpha T}(\bm{x}_i^{\alpha}-\bm{x}_{j}^{\alpha})+\frac{\rho}{2}\parallel\bm{x}_i^{\alpha}-\bm{x}_{j}^{\alpha} \parallel ^2)\\
			& \qquad+ \sum\limits_{i\in V}(\bm{\lambda}_{i,i}^{\alpha\beta T}(\bm{x}_i^{\alpha}-\bm{x}_{i}^{\beta})+\frac{\rho}{2}\parallel\bm{x}_i^{\alpha}-\bm{x}_{i}^{\beta} \parallel ^2). 
		\end{aligned}
	\end{equation}
	where $\bm{x}=[\bm{x}_1^{\alpha T}, \bm{x}_1^{\beta T}, \bm{x}_2^{\alpha T}, \bm{x}_2^{\beta T}, \ldots ,\bm{x}_N^{\alpha T},\bm{x}_N^{\beta T}]^T\in\mathbb{R}^{2Nn}$ is the augmented state. $\bm{\lambda}_{i,j}^{\alpha}$ is the Lagrange multiplier corresponding to the constraint $\bm{x}_{i}^{\alpha}=\bm{x}_{j}^{\alpha}$,   $\bm{\lambda}_{i,i}^{\alpha\beta}$ is the Lagrange multiplier corresponding to the constraint $\bm{x}_{i}^{\alpha}=\bm{x}_{i}^{\beta}$, and all $\bm{\lambda}_{i,j}^{\alpha}$ and  $\bm{\lambda}_{i,i}^{\alpha\beta}$ are stacked into $\bm{\lambda}$. $\rho$ is the penalty parameter, which is a positive constant scalar. It is worth noting that each agent $i$ does not need  to know the associated augmented Lagrangian function  (i.e., other agents' objective functions) to update its states $\bm{x}_i^{\alpha}$ and  $\bm{x}_i^{\beta}$, as shown below  in \eqref{eq:decomposition admm 1} and \eqref{eq:decomposition admm 2}.
	
	Based on Jacobian update,  we can solve \eqref{eq:decomposition admm form} by applying the following iterations for $i=1,2,\ldots,N$:
	\begin{numcases}{}
		\bm{x}_i^{\alpha (k+1)}=\underset{\bm{x}_i^{\alpha}}{\operatorname{argmin}}\frac{\gamma_i^{\alpha}\rho}{2}\parallel \bm{x}_i^{\alpha}-\bm{x}_i^{\alpha k} \parallel^2 \nonumber \\
		\qquad\qquad+\mathcal{L}_\rho^{k+1}(\bm{x}_1^{\alpha k},\bm{x}_1^{\beta k},\ldots,\bm{x}_i^{\alpha},\bm{x}_{i}^{\beta k},\ldots,\bm{x}_N^{\alpha k},\bm{x}_N^{\beta k},\bm{\lambda}^k) \nonumber \\
		\qquad\qquad=\underset{\bm{x}_i^{\alpha}}{\operatorname{argmin}} f_i^{\alpha (k+1)}(\bm{x}_i^{\alpha})+\frac{\gamma_i^{\alpha}\rho}{2}\parallel \bm{x}_i^{\alpha}-\bm{x}_i^{\alpha k}\parallel^2 \nonumber\\
		\qquad\qquad+\sum\limits_{j\in \mathcal{N}_{i}}(\bm{\lambda}_{i,j}^{\alpha k T}(\bm{x}_i^{\alpha}-\bm{x}_j^{\alpha k})+\frac{\rho}{2}\parallel\bm{x}_i^{\alpha}-\bm{x}_j^{\alpha k}\parallel^2) \nonumber \\
		\qquad\qquad+\bm{\lambda}_{i,i}^{\alpha\beta kT}(\bm{x}_i^{\alpha}-\bm{x}_i^{\beta k})+\frac{\rho}{2}\parallel\bm{x}_i^{\alpha}-\bm{x}_i^{\beta k}\parallel^2, \label{eq:decomposition admm 1}\\
		\bm{x}_i^{\beta (k+1)}=\underset{\bm{x}_i^{\beta}}{\operatorname{argmin}}\frac{\gamma_i^{\beta}\rho}{2}\parallel \bm{x}_i^{\beta}-\bm{x}_i^{\beta k} \parallel^2 \nonumber \\
		\qquad\qquad+\mathcal{L}_\rho^{k+1}(\bm{x}_1^{\alpha k},\bm{x}_1^{\beta k},\ldots,\bm{x}_i^{\alpha k},\bm{x}_{i}^{\beta},\ldots,\bm{x}_N^{\alpha k},\bm{x}_N^{\beta k},\bm{\lambda}^k) \nonumber \\
		\qquad\qquad= \underset{\bm{x}_i^{\beta}}{\operatorname{argmin}} f_i^{\beta(k+1)}(\bm{x}_i^{\beta})+\frac{\gamma_i^{\beta}\rho}{2}\parallel \bm{x}_i^{\beta}-\bm{x}_i^{\beta k}\parallel^2 \nonumber\\
		\qquad\qquad+\bm{\lambda}_{i,i}^{\beta\alpha kT}(\bm{x}_i^{\beta}-\bm{x}_i^{\alpha k})+\frac{\rho}{2}\parallel\bm{x}_i^{\beta}-\bm{x}_i^{\alpha k}\parallel^2, \label{eq:decomposition admm 2}\\
		\bm{\lambda}_{i,j}^{\alpha(k+1)}=\bm{\lambda}_{i,j}^{\alpha k}+\rho(\bm{x}_i^{\alpha(k+1)}-\bm{x}_j^{\alpha(k+1)}), \quad\forall j\in\mathcal{N}_{i}\label{eq:decomposition admm 3}\\
		\bm{\lambda}_{i,i}^{\alpha\beta(k+1)}=\bm{\lambda}_{i,i}^{\alpha\beta k}+\rho(\bm{x}_i^{\alpha(k+1)}-\bm{x}_i^{\beta(k+1)}), \label{eq:decomposition admm 4}\\
		\bm{\lambda}_{i,i}^{\beta\alpha(k+1)}=\bm{\lambda}_{i,i}^{\beta\alpha k}+\rho(\bm{x}_i^{\beta(k+1)}-\bm{x}_i^{\alpha(k+1)}). \label{eq:decomposition admm 5} 
	\end{numcases}
	Here, similar to our prior work \cite{zhang2018admm} and  algorithm \eqref{eq:conventional admm 1}-\eqref{eq:conventional admm 2}, both $\bm{\lambda}_{i,j}^{\alpha}$ and $\bm{\lambda}_{j,i}^{\alpha}$ are introduced for  the constraint $\bm{x}_{i}^{\alpha}=\bm{x}_{j}^{\alpha}, e_{i,j}\in E$  in \eqref{eq:decomposition admm 1}-\eqref{eq:decomposition admm 5} to unify the algorithm description.   Similarly, both $\bm{\lambda}_{i,i}^{\alpha\beta}$ and $\bm{\lambda}_{i,i}^{\beta\alpha}$ are introduced for the constraint $\bm{x}_{i}^{\alpha}=\bm{x}_{i}^{\beta}$ in \eqref{eq:decomposition admm 1}-\eqref{eq:decomposition admm 5} to unify the algorithm description.  
	
	Next we give in detail our privacy-preserving function-decomposition based algorithm.
	
	\begin{flushleft}
		\vspace{-0.85cm}
		\rule{0.49\textwidth}{0.2pt}
	\end{flushleft}
	\vspace{-0.2cm}
	
	\textbf{Algorithm \uppercase\expandafter{\romannumeral1}}
	\begin{flushleft}
		\vspace{-0.85cm}
		\rule{0.49\textwidth}{0.2pt}
	\end{flushleft}
	
	\textbf{Initial Setup:}
	For all $i=1,2,\ldots,N$, agent $i$ initializes $\bm{x}_i^{\alpha 0}$ and $\bm{x}_i^{\beta 0}$, and exchanges $\bm{x}_i^{\alpha 0}$ with neighboring agents. Then agent $i$  sets $\bm{\lambda}_{i,j}^{\alpha  0}=\bm{x}_i^{\alpha 0}-\bm{x}_j^{\alpha 0}$, $\bm{\lambda}_{i,i}^{\alpha\beta 0}=\bm{x}_i^{\alpha 0}-\bm{x}_i^{\beta 0}$, and $\bm{\lambda}_{i,i}^{\beta\alpha 0}=\bm{x}_i^{\beta 0} -\bm{x}_i^{\alpha 0}$.
	
	\textbf{Input:} $\bm{x}_i^{\alpha k}$, $\bm{\lambda}_{i,j}^{\alpha k}$, $\bm{\lambda}_{i,i}^{\alpha\beta k}$, $\bm{x}_i^{\beta k}$, $\bm{\lambda}_{i,i}^{\beta\alpha k}$.
	
	\textbf{Output:} $\bm{x}_i^{\alpha (k+1)}$, $\bm{\lambda}_{i,j}^{\alpha (k+1)}$, $\bm{\lambda}_{i,i}^{\alpha\beta (k+1)}$, $\bm{x}_i^{\beta (k+1)}$, $\bm{\lambda}_{i,i}^{\beta\alpha (k+1)}$.
	\begin{enumerate}	
		\item For all $i=1,2,\ldots,N$, agent $i$ constructs $f_i^{\alpha (k+1)}$ and $f_i^{\beta (k+1)}$ under the constraint $f_i =f_i^{\alpha (k+1)}+ f_i^{\beta (k+1)}$;
		\item For all $i=1,2,\ldots,N$, agent  $i$  updates $\bm{x}_i^{\alpha (k+1)}$ and  $\bm{x}_i^{\beta (k+1)}$ according to the update rules in \eqref{eq:decomposition admm 1}  and \eqref{eq:decomposition admm 2}, respectively;
		\item For all $i=1,2,\ldots,N$, agent  $i$  sends $\bm{x}_i^{\alpha (k+1)}$ to neighboring agents;
		\item For all $i=1,2,\ldots,N$, agent  $i$ computes $\bm{\lambda}_{i,j}^{\alpha (k+1)}$, $\bm{\lambda}_{i,i}^{\alpha\beta (k+1)}$ and $\bm{\lambda}_{i,i}^{\beta\alpha (k+1)}$  according to \eqref{eq:decomposition admm 3}-\eqref{eq:decomposition admm 5};
		\item Set $k$ to $k+1$, and go to 1).
	\end{enumerate}
	
	\begin{flushleft}
		\vspace{-0.9cm}
		\rule{0.49\textwidth}{0.2pt}
	\end{flushleft}

\section{Privacy Analysis}
In this section, we rigorously prove that each agent's  gradient of local objective function $\triangledown f_j$ cannot be inferred by honest-but-curious  adversaries and external eavesdroppers.

\begin{theorem} \label{theorem:privacy}
	In Algorithm \uppercase\expandafter{\romannumeral1}, agent $j$'s gradient of local objective function $\triangledown f_j$ at any point except the optimal solution will not be revealed to an honest-but-curious agent $i$.
\end{theorem}
\begin{proof}
	Suppose that an honest-but-curious adversary agent $i$ collects information from $K$ iterations to infer the gradient $\triangledown f_j$ of a neighboring agent $j$. The adversary agent $i$ can establish $2nK$ equations relevant to $\triangledown f_j$ by making use of the fact that the update rules of \eqref{eq:decomposition admm 1} and \eqref{eq:decomposition admm 2}   are publicly known, i.e.,
	\begin{equation}
	\left\{
	\begin{aligned}
	\lefteqn{ \triangledown f_j^{\alpha 1}(\bm{x}_{j}^{\alpha 1})+(\gamma_j^{\alpha}+D_{j}+1)\rho\bm{x}_j^{\alpha 1}-\gamma_j^{\alpha}\rho\bm{x}_j^{\alpha 0}}\\
	&\qquad\qquad\quad+\sum\limits_{m\in \mathcal{N}_{j}}(\bm{\lambda}_{j,m}^{\alpha 0}-\rho\bm{x}_m^{\alpha 0})
	+\bm{\lambda}_{j,j}^{\alpha\beta 0}-\rho\bm{x}_j^{\beta 0}=\bm{0}, \label{eq:j} \\		
	\lefteqn{ \triangledown f_j^{\beta 1}(\bm{x}_{j}^{\beta 1})+(\gamma_j^{\beta}+1)\rho\bm{x}_j^{\beta 1}-\gamma_j^{\beta}\rho\bm{x}_j^{\beta 0}+\bm{\lambda}_{j,j}^{\beta\alpha 0}-\rho\bm{x}_j^{\alpha 0}=\bm{0},} \\		
	&\qquad\qquad\qquad\qquad\vdots \\
	\lefteqn{ \triangledown f_j^{\alpha K}(\bm{x}_{j}^{\alpha K})+(\gamma_j^{\alpha}+D_{j}+1)\rho\bm{x}_j^{\alpha K}-\gamma_j^{\alpha}\rho\bm{x}_j^{\alpha (K-1)}+}\\
	&\sum\limits_{m\in \mathcal{N}_{j}}(\bm{\lambda}_{j,m}^{\alpha (K-1)}-\rho\bm{x}_m^{\alpha (K-1)})
	+\bm{\lambda}_{j,j}^{\alpha\beta (K-1)}-\rho\bm{x}_j^{\beta (K-1)}=\bm{0},\\
	\lefteqn{ \triangledown f_j^{\beta K}(\bm{x}_{j}^{\beta K})+(\gamma_j^{\beta}+1)\rho\bm{x}_j^{\beta K}-\gamma_j^{\beta}\rho\bm{x}_j^{\beta (K-1)}} \\
	&\qquad\qquad\qquad\qquad\qquad\quad+\bm{\lambda}_{j,j}^{\beta\alpha (K-1)}-\rho\bm{x}_j^{\alpha (K-1)}=\bm{0}. \\	
	\end{aligned}
	\right.
	\end{equation}
	In the system of $2nK$ equations \eqref{eq:j}, $\triangledown f_j^{\alpha k}(\bm{x}_{j}^{\alpha k})$ $(k=1,2,\ldots,K)$, $\triangledown f_j^{\beta k}(\bm{x}_{j}^{\beta k})$ $(k=1,2,\ldots,K)$, $\gamma_j^{\alpha}$, $\gamma_j^{\beta}$, and $\bm{x}_{j}^{\beta k}$ $(k=0,1,2,\ldots,K)$ are unknown to adversary agent $i$. Parameters $\bm{x}_m^{\alpha k}, m\neq j$ and $\bm{\lambda}_{j,m}^{\alpha k}, m\neq j$ are known to adversary agent $i$ only when agent $m$ and agent $i$ are neighbors.  So the above system of $2nK$ equations contains at least $3nK+n+2$ unknown variables, and adversary agent $i$ cannot infer the gradient of local objective function $\triangledown f_j$  by solving \eqref{eq:j}. 
	
	It is worth noting that after the optimization algorithm converges, adversary agent $i$ can have another piece of information according to the KKT conditions \cite{deng2017parallel}:
	\begin{equation} \label{eq:optimal}
	\begin{aligned}
	\triangledown f_j(\bm{x}_{j}^{*})=-\sum_{m\in\mathcal{N}_j}\bm{\lambda}_{j,m}^{\alpha*}.
	\end{aligned}
	\end{equation} 
	If agent $j$'s neighbors are also neighbors to agent $i$, the exact gradient of $f_j$ at the optimal solution can be inferred by an honest-but-curious agent $i$. Therefore, agent $j$'s gradient of local objective function $\triangledown f_j$  will not be revealed to an honest-but-curious agent $i$  at any point except the optimal solution.
\end{proof}

\begin{Corollary 1}
	In Algorithm \uppercase\expandafter{\romannumeral1}, agent $j$'s gradient of local objective function $\triangledown f_j$ at any point except the optimal solution will not be revealed to external eavesdroppers.
\end{Corollary 1}
\begin{proof}
	The proof can be obtained following a similar line of reasoning of Theorem \ref{theorem:privacy}. External eavesdroppers can also establish a system of $2nK$ equations \eqref{eq:j} to infer agent $j$'s gradient $\triangledown f_j$. However, the number of unknowns $\triangledown f_j^{\alpha k}(\bm{x}_{j}^{\alpha k})$ $(k=1,2,\ldots,K)$, $\triangledown f_j^{\beta k}(\bm{x}_{j}^{\beta k})$ $(k=1,2,\ldots,K)$, $\gamma_j^{\alpha}$, $\gamma_j^{\beta}$, and $\bm{x}_{j}^{\beta k}$ $(k=0,1,2,\ldots,K)$ adds up to $3nk+n+2$, making the system of equations established by the external eavesdropper undetermined. Therefore, external eavesdroppers cannot infer  the gradient of local objective function $\triangledown f_j$  at any point except the optimal solution.  
\end{proof}

\begin{Remark 1}
	It is worth noting that if multiple adversary agents cooperate to infer the information of agent $j$, they can only establish a system of $2nK$ equations containing at least $3nK+n+2$ unknown variables as well. Therefore, our algorithm can protect the privacy of agents against multiple honest-but-curious  adversaries and external eavesdroppers.
\end{Remark 1}

\section{Convergence Analysis}
	In this section,  we rigorously prove the convergence of Algorithm  \uppercase\expandafter{\romannumeral1} under the following assumptions:
	\begin{Assumption 1} \label{assumption 1}
		Each local function $f_i: \mathbb{R}^n\to\mathbb{R}$ is strongly convex and continuously differentiable, i.e., $$(\triangledown f_i(\tilde{\bm{x}})-\triangledown f_i(\tilde{\bm{y}}))^T(\tilde{\bm{x}}-\tilde{\bm{y}})\ge m_{i}\parallel\tilde{\bm{x}}-\tilde{\bm{y}}\parallel^2.$$ In addition, there exists a lower bound $m_f>0$ such that $m_{i}\ge2m_f, \forall i=\{1,2,\ldots,N\}$  is true.
	\end{Assumption 1}
	\begin{Assumption 1} \label{assumption 2}
		Each private local function $f_i: \mathbb{R}^n\to\mathbb{R}$ has Lipschitz continuous  gradients, i.e., $$\parallel\triangledown f_i(\tilde{\bm{x}})-\triangledown f_i(\tilde{\bm{y}})\parallel\le L_i\parallel\tilde{\bm{x}}-\tilde{\bm{y}}\parallel.$$ % In addition, there is a $L<+\infty$ such that $L_i\le L, \forall i=\{1,2,\ldots,N\}$  is true.
	\end{Assumption 1}
	\begin{Assumption 1} \label{assumption gi}
		$f_i^{\alpha k}$ is chosen under the following conditions:
		
		1)  $f_i^{\alpha k}$  is strongly convex and differentiable, i.e.,  $$(\triangledown f_i^{\alpha k}(\tilde{\bm{x}})-\triangledown f_i^{\alpha k}(\tilde{\bm{y}}))^T(\tilde{\bm{x}}-\tilde{\bm{y}})\ge m_f\parallel\tilde{\bm{x}}-\tilde{\bm{y}}\parallel^2.$$
				
		2)  $f_i^{\beta k}=f_i-f_i^{\alpha k}$  is strongly convex and differentiable, i.e.,  $$(\triangledown f_i^{\beta k}(\tilde{\bm{x}})-\triangledown f_i^{\beta k}(\tilde{\bm{y}}))^T(\tilde{\bm{x}}-\tilde{\bm{y}})\ge m_f\parallel\tilde{\bm{x}}-\tilde{\bm{y}}\parallel^2.$$
		
		3) $f_i^{\alpha k}$  has Lipschitz continuous  gradients, i.e.，  there exists an $L<+\infty$ such that
			$$\parallel\triangledown f_i^{\alpha k}(\tilde{\bm{x}})-\triangledown f_i^{\alpha k}(\tilde{\bm{y}})\parallel\le L\parallel\tilde{\bm{x}}-\tilde{\bm{y}}\parallel.$$
		
		4) $f_i^{\beta k}=f_i-f_i^{\alpha k}$  has Lipschitz continuous  gradients, i.e.， there exists an $L<+\infty$ such that
		 $$\parallel\triangledown f_i^{\beta k}(\tilde{\bm{x}})-\triangledown f_i^{\beta k}(\tilde{\bm{y}})\parallel\le L\parallel\tilde{\bm{x}}-\tilde{\bm{y}}\parallel.$$
		 
		 5) $\lim\limits_{k\to\infty}f_i^{\alpha k}\to f_i^{\alpha*}$ and $f_i^{\alpha k}(\tilde{\bm{x}})$ is bounded when $\tilde{\bm{x}}$ is bounded.		 
	\end{Assumption 1}
	
	It is worth noting that under Assumption \ref{assumption 1} and Assumption \ref{assumption 2}, $f_i^{\alpha k}$  can be easily designed to meet Assumption \ref{assumption gi}. A quick example is $f_i^{\alpha k}(\tilde{\bm{x}})=\frac{m_f }{2}\tilde{\bm{x}}^T\tilde{\bm{x}}+\bm{b}_i^{kT}\tilde{\bm{x}}$ where $\bm{b}_i^k\in \mathbb{R}^n$ can be time-varying, and satisfies $\lim\limits_{k\to \infty}\bm{b}_i^k\to \bm{b}_i^*$ and $-\infty<\parallel\bm{b}_i^k\parallel<\infty$.
	
	Because the function decomposition process amounts to converting the original network to  a virtual network $G'=(V',E')$ of $2N$ agents, as shown in Fig. \ref{fig:G2N}, we analyze the convergence of our algorithm based on the virtual network $G'=(V',E')$. To simplify and unify  the notations, we relabel the local objective functions  $f_i^{\alpha k}$ and $f_i^{\beta k}$ for all $i=1,2,\ldots,N$ as $h_1^k, h_2^k,\dots,h_{2N}^k$.  We relabel the associated states $\bm{x}_i^{\alpha k}$ and $\bm{x}_i^{\beta k}$ for all $i=1,2,\dots,N$  as $\bm{x}_1^k, \bm{x}_2^k,\dots,\bm{x}_{2N}^k$.  In addition, we relabel parameters $\gamma_i^{\alpha}$ and $\gamma_i^{\beta}$ for all $i=1,2,\dots,N$  correspondingly as $\gamma_1, \gamma_2,\ldots,\gamma_{2N}$.  Then problem \eqref{eq:decomposition admm form} can be rewritten as
		\begin{figure}
			\begin{center}
				\subfloat[$G=(V,E)$ of $N$ agents]{\includegraphics[width=0.185\textwidth]{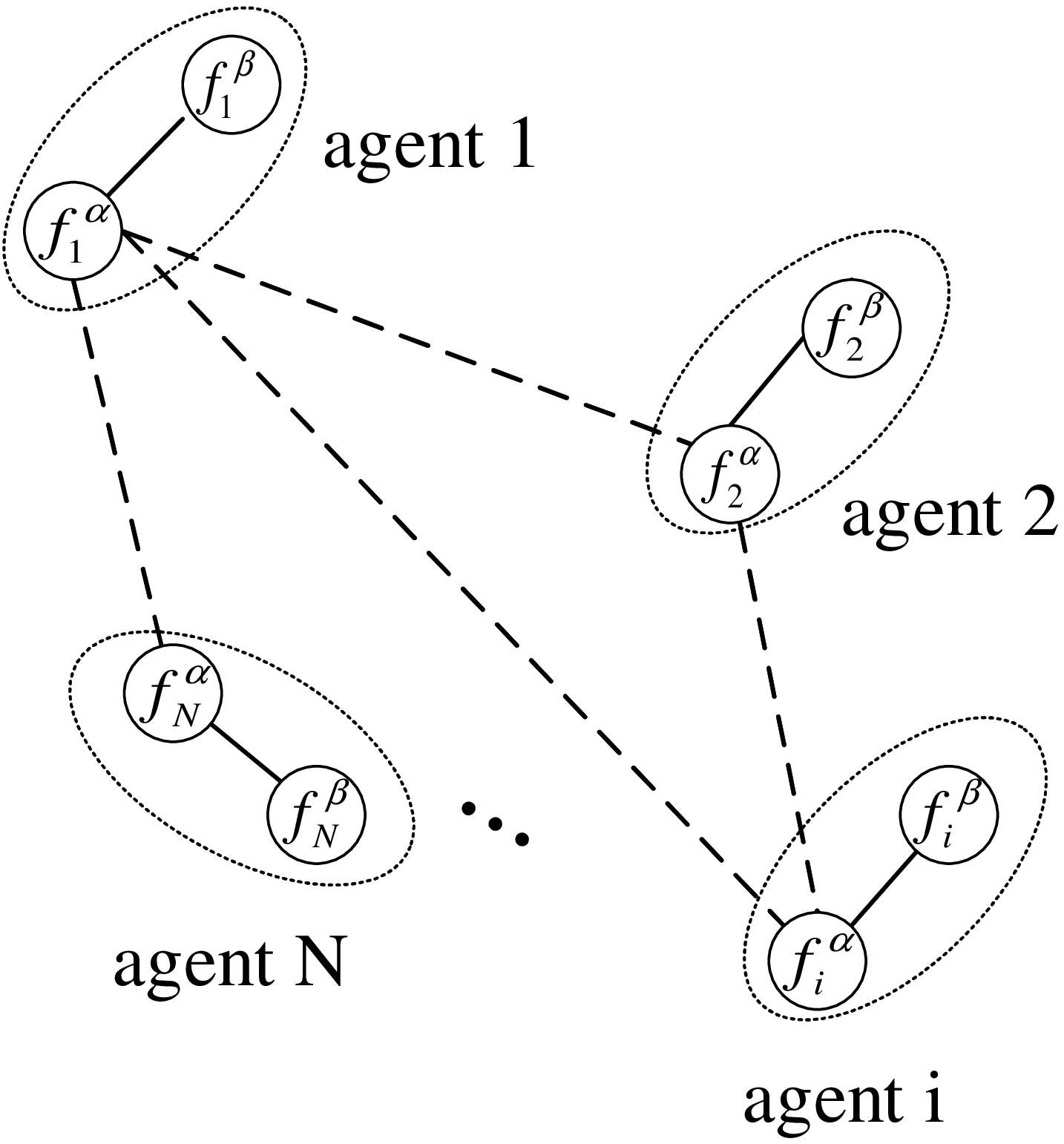}%
					\label{fig: G}}
				\hspace{.3in} 
				\subfloat[Virtual network $G'=(V',E')$ of $2N$ agents]{\includegraphics[width=0.245\textwidth]{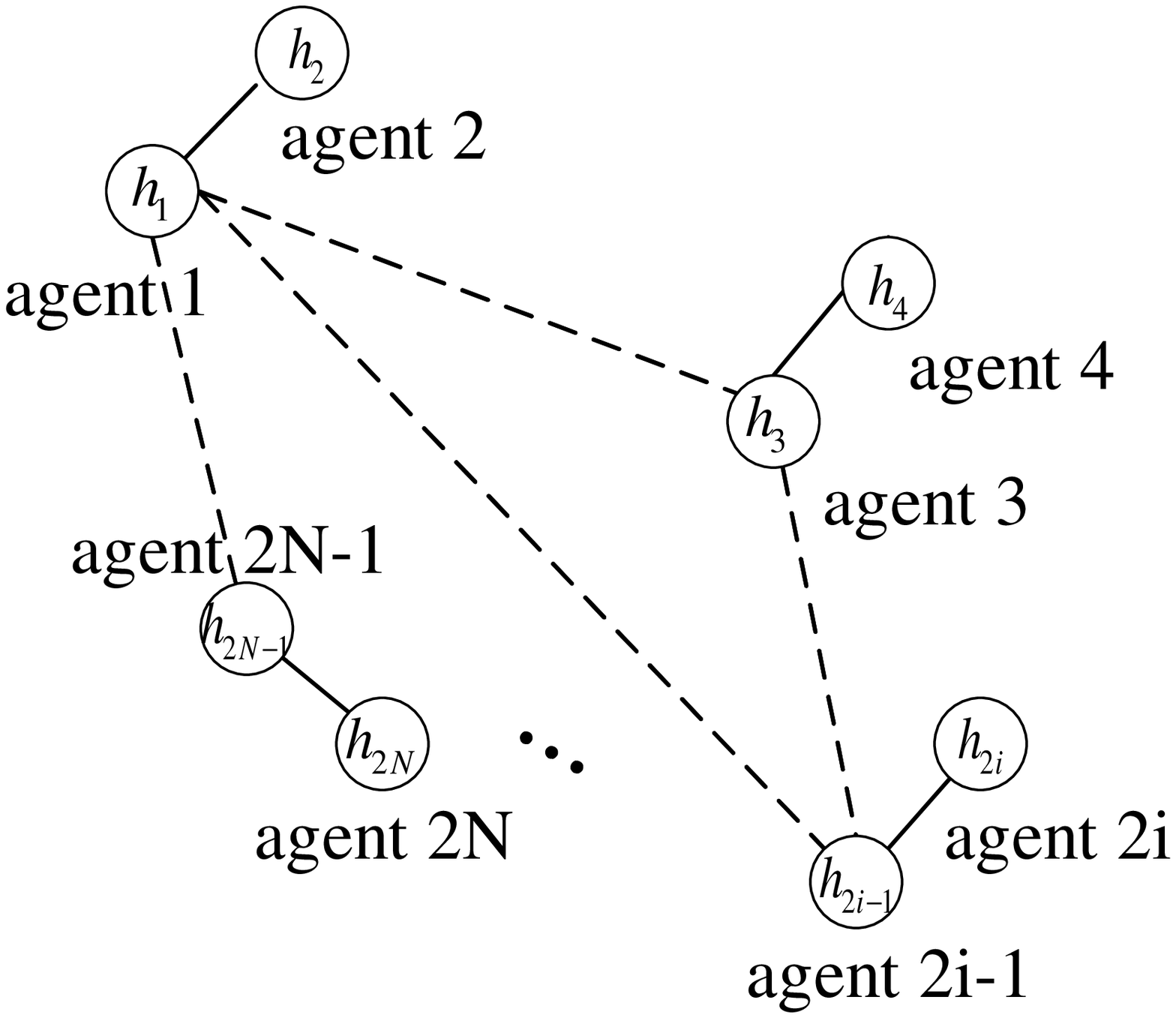}%
					\label{fig: G'}}
			\end{center}
			\caption{Function-decomposition based privacy-preserving decentralized optimization equals to converting the original network into  a virtual network $G'=(V',E')$ of $2N$ agents.}
			\label{fig:G2N}
		\end{figure}

	\begin{equation}\label{eq:decomposition admm form 2}
		\begin{aligned}
		&\mathop {\min }\limits_{\bm{x}_{i}\in\mathbb{R}^n,\,i\in\{1,2,\ldots,2N\}} \qquad \sum\limits_{i=1}^{2N}h_i^{k}(\bm{x}_{i})\\
		&\qquad\textrm{subject to} \qquad\qquad A\bm{x}=\bm{0}
		\end{aligned}
	\end{equation}
	where $\bm{x}=[\bm{x}_1^T, \bm{x}_2^T,\ldots,\bm{x}_{2N}^T]^T\in\mathbb{R}^{2Nn}$ and $A=[a_{m,l}]\otimes I_n\in\mathbb{R}^{|E'|n\times 2Nn}$ is the edge-node incidence matrix of graph $\mathcal{G}'$ as defined in \cite{wei2012distributed}. Parameter $a_{m,l}$  is defined as
	\begin{eqnarray} \nonumber
		a_{m,l}=\left\{\begin{matrix}
		1& \textrm {if the } m^{th} \textrm{ edge originates from agent }  l,  \\
		-1& \textrm {if the } m^{th} \textrm{ edge terminates at agent }  l, \\
		0& \textrm{otherwise}.
		\end{matrix}\right.
	\end{eqnarray}
	We define each edge $e_{i,j}$ originating from $i$ and terminating at $j$ and denote an edge as $e_{i,j}\in E'$ if $i<j$ is true or $e_{j,i}\in E'$ otherwise. 
	
	Denote the iterating results in the $k$th step in Algorithm \uppercase\expandafter{\romannumeral1} as follows:
	\begin{displaymath}
		\begin{aligned}
		&\bm{x}^k=[\bm{x}_1^{kT},\bm{x}_2^{kT},\ldots,\bm{x}_{2N}^{kT} ]^T\in\mathbb{R}^{2Nn}, \\
		&\bm{\lambda}^k=[\bm{\lambda}_{i,j}^k]_{ij,e_{i,j}\in E'}\in\mathbb{R}^{|E'|n}, \\
		&\bm{y}^k=[\bm{x}^{kT},\bm{\lambda}^{kT}]^T\in\mathbb{R}^{(|E'|+2N)n}
		\end{aligned}
	\end{displaymath}
	Further augment the coefficients $\gamma_i$ $(i=1,2,\ldots,2N)$  into the matrix form
	\begin{displaymath}
		U={\rm diag}\{\gamma_1,\gamma_2,\ldots,\gamma_{2N}\}\otimes I_n\in\mathbb{R}^{2Nn\times 2Nn},
	\end{displaymath}
	and $D_i$ into the matrix form
	\begin{displaymath}
		D={\rm diag}\{D_1,D_2,\ldots,D_{2N}\}\otimes I_n\in\mathbb{R}^{2Nn\times 2Nn}.
	\end{displaymath}
	
	Now we are in position to give the main results for this section:
	\begin{Lemma 1} \label{lemma 1}
		Let $\bm{x}^*$ be the optimal solution, $\bm{\lambda}^{k*}$ be the optimal multiplier to \eqref{eq:decomposition admm form 2} at iteration $k$, and $\bm{y}^{k*}=[\bm{x}^{*T},\bm{\lambda}^{k*T}]^T$. Further define $Q=U+D-A^TA$, $H={\rm diag}\{\rho Q, \frac{1}{\rho}I_{|E'|n}\}$, and let $u>1$ be an arbitrary constant, then we have
		 \begin{equation} \label{eq:lemma 1}
			 \begin{aligned}
			 \parallel \bm{y}^{k+1}-\bm{y}^{k+1*}\parallel_H\le\frac{ \parallel \bm{y}^{k}-\bm{y}^{k+1*}\parallel_H}{\sqrt{1+\delta}}
			 \end{aligned}
		 \end{equation}
		  if $U+D - A^TA$ is positive semi-definite and Assumptions \ref{assumption 1},  \ref{assumption 2}, and  \ref{assumption gi} are satisfied. In \eqref{eq:lemma 1},  $\parallel \tilde{\bm{x}}\parallel_H=\sqrt{\tilde{\bm{x}}^TH\tilde{\bm{x}}}$ and
		  \begin{equation}
			  \begin{aligned}
			  \delta=\min\{\frac{(u-1)A_{\min}}{u Q_{\max}}, \frac{2m_fA_{\min}\rho}{uL^2+\rho^2A_{\min}Q_{\max}}\},
			  \end{aligned}
		  \end{equation}
		  where $Q_{\max}$ is the largest eigenvalue of $Q$, $A_{\min}$ is the smallest nonzero eigenvalue of $A^TA$, $m_f$ is the strongly convexity modulus, and $L$ is the Lipschitz modulus. 
	\end{Lemma 1}
	\begin{proof}
		The results can be obtained following a similar line of reasoning in \cite{ling2014decentralized}. The detailed proof is given in the supplementary materials and can be found online \cite{zhang2018privacy}. 
	\end{proof}
	
	\begin{Lemma 1} \label{lemma 2}
		Let $\bm{x}^*$ be the optimal solution, $\bm{\lambda}^{k*}$ be the optimal multiplier to \eqref{eq:decomposition admm form 2} at iteration $k$, and $\bm{y}^{k*}=[\bm{x}^{*T},\bm{\lambda}^{k*T}]^T$. Further define $Q=U+D-A^TA$ and $H={\rm diag}\{\rho Q, \frac{1}{\rho}I_{|E'|n}\}$, then we have
		\begin{equation} \label{eq:lemma 2}
			\begin{aligned}
			\parallel \bm{y}^{k}-\bm{y}^{k+1*}\parallel_H\le\parallel \bm{y}^{k}-\bm{y}^{k*}\parallel_H+p(k)
			\end{aligned}
		\end{equation}
	    if $U+D - A^TA$ is positive semi-definite and Assumptions \ref{assumption 1},  \ref{assumption 2}, and  \ref{assumption gi} are satisfied. In \eqref{eq:lemma 2},
		\begin{equation}
			\begin{aligned}
		    p(k)=\frac{1}{\sqrt{\rho A_{\min}}}\parallel\triangledown h^{k+1}(\bm{x}^*)-\triangledown h^k(\bm{x}^*) \parallel
			\end{aligned}
		\end{equation}
		where $h^k(\bm{x})=\sum\limits_{i=1}^{2N}h_i^{k}(\bm{x}_{i})$.
	\end{Lemma 1}
	\begin{proof}
		The results can be obtained following a similar line of reasoning in \cite{ling2014decentralized}. The detailed proof is given in the supplementary materials and can be found online  \cite{zhang2018privacy}. 
	\end{proof}
	
	\begin{Lemma 1} \label{lemma 3}
		Let $\bm{x}^*$ be the optimal solution, $\bm{\lambda}^{k*}$ be the optimal multiplier to \eqref{eq:decomposition admm form 2} at iteration $k$, and $\bm{y}^{k*}=[\bm{x}^{*T},\bm{\lambda}^{k*T}]^T$. Further define $Q=U+D-A^TA$ and $H={\rm diag}\{\rho Q, \frac{1}{\rho}I_{|E'|n}\}$, then we have
		 \begin{equation} \label{eq:lemma 3}
			 \begin{aligned}
			 \parallel \bm{y}^{k+1}-\bm{y}^{k+1*}\parallel_H\le\frac{ \parallel \bm{y}^{k}-\bm{y}^{k*}\parallel_H}{\sqrt{1+\delta}}+\frac{ p(k)}{\sqrt{1+\delta}}
			 \end{aligned}
		 \end{equation}
		if $U+D - A^TA$ is positive semi-definite and Assumptions \ref{assumption 1},  \ref{assumption 2}, and   \ref{assumption gi} are satisfied.
	\end{Lemma 1}
	\begin{proof}
		Combining \eqref{eq:lemma 1} and  \eqref{eq:lemma 2}, we obtain the result directly.
	\end{proof}
	
	Lemma \ref{lemma 3} indicates that $\parallel \bm{y}^{k+1}-\bm{y}^{k+1*}\parallel_H$ converges linearly to a neighborhood of $0$.
	
	\begin{theorem}\label{Theorem:ADMM Convergence}	
		Algorithm \uppercase\expandafter{\romannumeral1} is guaranteed to converge to the optimal solution to \eqref{eq:decomposition admm form 2} if $U+D - A^TA$ is positive semi-definite and Assumption \ref{assumption 1}, Assumption \ref{assumption 2}, and Assumption  \ref{assumption gi} are satisfied. 
	\end{theorem}	
	\begin{proof}
		The proof is provided in the Appendix. 
	\end{proof}

\begin{Remark 1}
	It is worth noting that problem \eqref{eq:decomposition admm form 2} is a reformulation of problem \eqref{eq:global function}. So Theorem \ref{Theorem:ADMM Convergence} guarantees that each agent's state will converge to the optimal solution to \eqref{eq:global function}. 
\end{Remark 1}

\section{Numerical Experiments}
We first present a numerical example to illustrate the efficiency of the proposed approach. Then we compare our approach with the differential-privacy based algorithm in \cite{huang2015differentially}. We conducted numerical experiments on the following optimization problem.
\begin{eqnarray}
	\min\limits_{\tilde{\bm{x}}}\qquad \sum\limits_{i=1}^{N}\parallel \tilde{\bm{x}}-\bm{y}_i\parallel^2 \label{eq:simulation}
\end{eqnarray}
with $\bm{y}_i\in\mathbb{R}^n$. Each agent $i$  deals with a private local objective function
\begin{eqnarray} \label{eq:individual function}
	f_i(\bm{x}_i)=\parallel \bm{x}_i-\bm{y}_i\parallel^2, \forall i\in\{1,2,\ldots,N\}. 
\end{eqnarray}
We used the above optimization problem \eqref{eq:simulation} because it is easy to verify whether the obtained value  is the optimal solution, which should be $\frac{\sum_{i=1}^{N}\bm{y}_i}{N}$. Furthermore, \eqref{eq:simulation} makes it easy to compare with \cite{huang2015differentially}, whose simulation is also based on \eqref{eq:simulation}. 
	
\subsection{Evaluation of Our Approach}
	To solve the optimization problem \eqref{eq:simulation}, $f_i^{\alpha k}(\tilde{\bm{x}})$ was set to $f_i^{\alpha k}(\tilde{\bm{x}})=\frac{1}{2}\tilde{\bm{x}}^T\tilde{\bm{x}}+(\bm{b}_i^k)^{T}\tilde{\bm{x}}$ for our approach in the simulations,  where $\bm{b}_i^k$ was set to $\bm{b}_i^k=\frac{1}{k+1}\bm{c}_i+\bm{d}_i$ with $\bm{c}_i \in \mathbb{R}^n$ and $\bm{d}_i \in \mathbb{R}^n$ being constants private to agent $i$. Fig. \ref{fig:x} visualizes the evolution of $\bm{x}_i^{\alpha}$ and $\bm{x}_i^{\beta}$ $(i=1,2,...,6)$  in one specific run where the network deployment is illustrated in Fig. \ref{fig:network structure illustration 1}.  All $\bm{x}_i^{\alpha}$ and $\bm{x}_i^{\beta}$ $(i=1,2,...,6)$ converged to the optimal solution $13.758$. 
	\begin{figure}[!htbp]
		\begin{center}
			\includegraphics[width=0.4\textwidth]{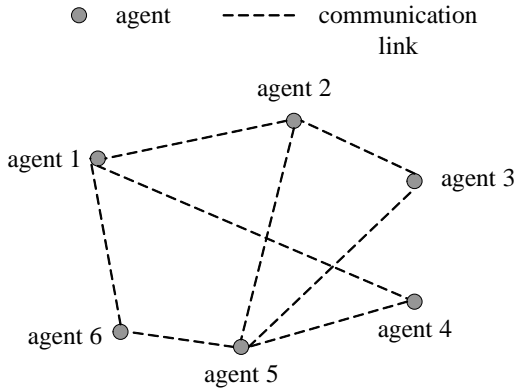}
		\end{center}
		\caption{A network of six agents ($N=6$).}
		\label{fig:network structure illustration 1}
	\end{figure}

	\begin{figure}[!htbp]
		\begin{center}
			\includegraphics[width=0.49\textwidth]{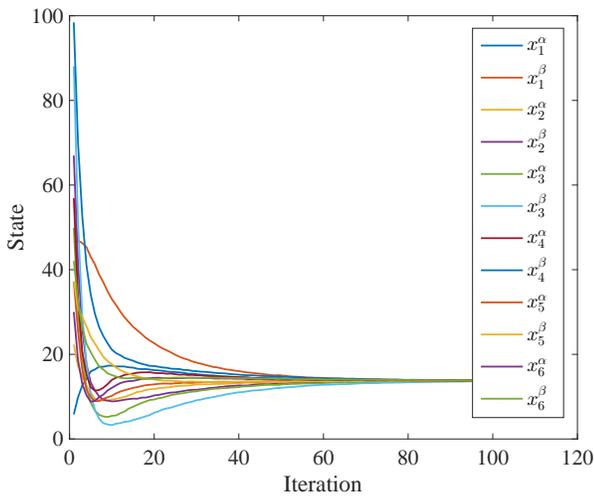}
		\end{center}
		\caption{The evolution of $\bm{x}_i^{\alpha}$ and $\bm{x}_i^{\beta}$ in our approach.}
		\label{fig:x}
	\end{figure}
	
\subsection{Comparison with the algorithm in \cite{huang2015differentially} }
	 Under the network deployment in Fig. \ref{fig:network structure illustration 1}, we compared our privacy-preserving approach with the differential-privacy based  algorithm in \cite{huang2015differentially}. We simulated the algorithm in \cite{huang2015differentially} under seven different privacy levels: $$\epsilon=0.2, 1, 10, 20, 30, 50, 100.$$ In the objective function  \eqref{eq:simulation}, $\bm{y}_i$ was set to $\bm{y}_i=[0.1\times (i-1)+0.1;0.1\times (i-1)+0.2]$. The domain of optimization for the algorithm in \cite{huang2015differentially} was set to $\mathcal{X}=\{(x,y)\in\mathbb{R}^2|x^2+y^2\le 1\}$. Note that the optimal solution $[0.35;0.45]$ resided in $\mathcal{X}$. 
	Detailed parameter settings for the algorithm in \cite{huang2015differentially} were given as $n=2$, $c=0.5$, $q=0.8$, $p=0.9$, and 
	\begin{eqnarray} \label{eq: aij}
		a_{ij}=\left\{
		\begin{aligned}
			0.2 \qquad\qquad\qquad &  j\in\mathcal{N}_i, \\
			0 \qquad\qquad\qquad &  j\notin\mathcal{N}_i, j\neq i, \\
			1-\sum\limits_{j\in\mathcal{N}_i}a_{ij} \qquad & i=j,
		\end{aligned}
		\right.
	\end{eqnarray} 
	for $i=1,2,...,6$. In addition, the performance index $d$ in \cite{huang2015differentially} was used to quantify the optimization error here, which was computed as the average value of squared distances with respect to the optimal solution over $M$ runs \cite{huang2015differentially}, i.e.,
	$$d=\frac{\sum\limits_{i=1}^{6}\sum\limits_{l=1}^{M}\parallel \bm{x}_i^l-[0.35;0.45]\parallel^2}{6M}.$$
	Here $\bm{x}_i^l$ is the obtained solution of agent $i$ in the $l$th run.  For our approach, $\bm{x}_i^l$ was calculated as the average of $\bm{x}_i^{\alpha l}$ and $\bm{x}_i^{\beta l}$.
	
	Simulation results from 5,000 runs showed that our approach  converged to $[0.35;0.45]$ with an error $d=5.1 \times10^{-4}$, which is negligible compared with the simulation results under the algorithm in  \cite{huang2015differentially} (cf. Fig. \ref{fig:d}, where each differential privacy level was implemented for 5,000 times). The results confirm the trade-off between privacy and accuracy in differential-privacy based approaches.    
	
	\begin{figure}[!htbp]
		\begin{center}
			\includegraphics[width=0.48\textwidth]{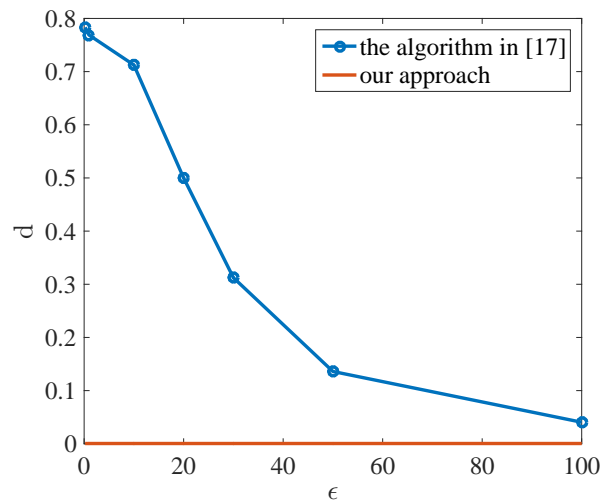}
		\end{center}
		\caption{The comparison of our approach with the algorithm in  \cite{huang2015differentially} in terms of optimization error.}
		\label{fig:d}
	\end{figure}
\section{Conclusions}
	In this paper, we proposed a novel  approach to enabling privacy-preservation in decentralized optimization based on function decomposition, which neither compromises the optimality of optimization nor relies on an aggregator or third party. Theoretical analysis confirms that an honest-but-curious adversary cannot infer the information of neighboring agents even by recording and analyzing the information exchanged in multiple iterations. In addition, our approach can also avoid an external eavesdropper from inferring the information of participating agents. Numerical simulation results confirmed the effectiveness of the proposed approach.

\section*{APPENDIX}
	\subsection{Proof of Theorem \ref{Theorem:ADMM Convergence}}
	From Lemma \ref{lemma 3}, we can obtain 
	\begin{equation}
		\begin{aligned}
		\sqrt{1+\delta}^k\parallel \bm{y}^{k}-\bm{y}^{k*}\parallel_H\le \parallel \bm{y}^{0}-\bm{y}^{0*}\parallel_H+ \sum_{s=0}^{k-1}\sqrt{1+\delta}^{s}p(s).
		\end{aligned}
	\end{equation}
	Dividing both sides by $\sqrt{1+\delta}^k$, we have
	\begin{equation} \label{eq:ylim0}
		\begin{aligned}
		\parallel \bm{y}^{k}-\bm{y}^{k*}\parallel_H\le \frac{\parallel\bm{y}^{0}-\bm{y}^{0*}\parallel_H}{\sqrt{1+\delta}^k}+ \sum_{s=0}^{k-1}\frac{1}{\sqrt{1+\delta}^{k-s}}p(s).
		\end{aligned}
	\end{equation}
	It is clear  $\lim\limits_{k\to\infty}\frac{\parallel\bm{y}^{0}-\bm{y}^{0*}\parallel_H}{\sqrt{1+\delta}^k}= 0$ due to $\delta>0$. Now our main goal is to prove $\lim\limits_{k\to\infty}\sum_{s=0}^{k-1}\frac{1}{\sqrt{1+\delta}^{k-s}}p(s)= 0$. Recall that in Assumption \ref{assumption gi} we have $\lim\limits_{k\to\infty} f_i^{\alpha k}\to f_i^{\alpha *}$. Therefore, we have $\lim\limits_{k\to\infty} h_i^k\to h_i^*$ for all $i=1,2,\dots, 2N$, i.e., $h_i^k$ converges to a fixed function $h_i^*$ ($h^k$ converges to a fixed function $h^*$). On the other hand, we have
	\begin{equation}
		\begin{aligned}
		p(k)=\frac{1}{\sqrt{\rho A_{\min}}}\parallel\triangledown h^{k+1}(\bm{x}^*)-\triangledown h^k(\bm{x}^*) \parallel.
		\end{aligned}
	\end{equation}
	
	As a result of the convergence of $h^k$, we have $\lim\limits_{k\to\infty}p(k)=0$. Therefore, we have that $p(k)$ is bounded, i,e., there exists a $B$ such that $p(k)\le B, \forall k$. In addition, we have
	$$\forall  \varepsilon_1>0, \quad \exists N_1\in\mathbb{N}^+, \text{\quad s.t.\quad} |p(k)|\le\varepsilon_1, \quad \forall k\ge N_1,$$
	where $\mathbb{N}^+$ is the set of positive integers.
	Further letting $\eta=\frac{1}{\sqrt{1+\delta}}$ and $F(k)=\sum_{s=0}^{k-1}\frac{1}{\sqrt{1+\delta}^{k-s}}p(s)$, we have $\eta\in(0,1)$ and
	
	 \begin{equation}
		\begin{aligned}
		 F(k)&=\sum_{s=0}^{k-1}\eta^{k-s}p(s) \\
		 &=\sum_{s=0}^{N_1}\eta^{k-s}p(s)+\sum_{s=N_1+1}^{k-1}\eta^{k-s}p(s) \\
		 &\le B\sum_{s=0}^{N_1}\eta^{k-s}+\varepsilon_1\sum_{s=N_1+1}^{k-1}\eta^{k-s} \\
		 &= B\eta^k\frac{\eta^{-N_1}-\eta}{1-\eta}+\varepsilon_1\frac{\eta-\eta^{k-N_1-1}}{1-\eta} \\
		 &\le B\eta^k\frac{\eta^{-N_1}-\eta}{1-\eta}+\varepsilon_1\frac{\eta}{1-\eta}
		\end{aligned}
	\end{equation}
	for $k\ge N_1+2$.
	
	Recalling $\eta\in(0,1)$, we have $\lim\limits_{k\to\infty}B\eta^k\frac{\eta^{-N_1}-\eta}{1-\eta}=0$ and
	 \begin{equation}
		 \begin{aligned}
		 &\forall  \varepsilon=\varepsilon_1>0, \quad \exists N_2\in\mathbb{N}^+, \\
		 &\text{s.t.\quad} |B\eta^k\frac{\eta^{-N_1}-\eta}{1-\eta}|\le\varepsilon_1, \quad \forall k\ge N_2,
		 \end{aligned}
	 \end{equation}
	
	Therefore, we can obtain
	 \begin{equation}
		 \begin{aligned}
		 &\forall  \varepsilon=\varepsilon_1>0, \quad \exists N=\max\{N_1,N_2\}, \\
		 &\text{s.t.\quad} |F(k)|\le\varepsilon_1+\varepsilon_1\frac{\eta}{1-\eta}=\frac{1}{1-\eta}\varepsilon_1, \quad \forall k\ge N,
		 \end{aligned}
	 \end{equation}
	 which proves that $\lim\limits_{k\to\infty}F(k)=0$. Then according to \eqref{eq:ylim0}, we have $\lim\limits_{k\to\infty}\parallel \bm{y}^{k}-\bm{y}^{k*}\parallel_H=0$. Since  $\parallel \bm{x}^{k}-\bm{x}^{*}\parallel_Q\le \parallel \bm{y}^{k}-\bm{y}^{k*}\parallel_H$, we have $\lim\limits_{k\to\infty}\parallel \bm{x}^{k}-\bm{x}^{*}\parallel_Q=0$ as well, which completes the proof.

\bibliographystyle{unsrt}
\bibliography{abbr_bibli}

\clearpage
\newpage
\onecolumn
\begin{center}
	\textbf{\large Supplemental Materials: Privacy-preserving Decentralized Optimization via Decomposition }
\end{center}
\subsection{Proof of Lemma \ref{lemma 1}}
According to the update rules in \eqref{eq:decomposition admm 1} and \eqref{eq:decomposition admm 2} , we have
\begin{equation}
\begin{aligned} \label{eq:subproblem}
\triangledown h_i^{k+1}(\bm{x}_i^{k+1})+\sum_{j\in\mathcal{N}_i}(\bm{\lambda}_{i,j}^k+\rho(\bm{x}_i^{k+1}-\bm{x}_j^k))
+\gamma_i\rho(\bm{x}_i^{k+1}-\bm{x}_i^k)=\bm{0}
\end{aligned}
\end{equation}
for all $i=1,2,\ldots,2N$.

Let $(\bm{x}^*, \bm{\lambda}^{k*})$ be the Karush-Kuhn-Tucker (KKT) points for  \eqref{eq:decomposition admm form 2} at iteration $k$, we have
\begin{equation}
\begin{aligned} \label{eq:kkt}
-A_i^T\bm{\lambda}^{k*}&=\triangledown h_i^k(\bm{x}_i^*) \\
A\bm{x}^*&=\bm{0}
\end{aligned}
\end{equation}
where $A_i$ indicates the columns of $A$ corresponding to agent $i$. It is worth noting that since $\sum\limits_{i=1}^{2N}h_i^{k}(\bm{x}_{i})$ is strongly convex, $\bm{x}^*$ is the optimal solution to  \eqref{eq:decomposition admm form 2}.

Since each $h_i^k$ is strongly convex, we have
\begin{equation}
\begin{aligned}
(\triangledown h_i^{k+1}(\bm{x}_i^{k+1})-\triangledown h_i^{k+1}(\bm{x}_i^*))^T(\bm{x}_i^{k+1}-\bm{x}_i^*)
\ge m_f\parallel\bm{x}_i^{k+1}-\bm{x}_i^*\parallel^2.
\end{aligned}
\end{equation}

Combining the above equation with \eqref{eq:subproblem} and \eqref{eq:kkt}, we have
\begin{equation}
\begin{aligned}
(-\sum_{j\in\mathcal{N}_i}(\bm{\lambda}_{i,j}^k+\rho(\bm{x}_i^{k+1}-\bm{x}_j^k))-\gamma_i\rho(\bm{x}_i^{k+1}-\bm{x}_i^k)
+A_i^T\bm{\lambda}^{k+1*})^T(\bm{x}_i^{k+1}-\bm{x}_i^*)\ge m_f\parallel\bm{x}_i^{k+1}-\bm{x}_i^*\parallel^2.
\end{aligned}
\end{equation}

Noting that $\bm{\lambda}_{i,j}^{k+1}=\bm{\lambda}_{i,j}^k+\rho(\bm{x}_i^{k+1}-\bm{x}_j^{k+1})$, one has
\begin{equation} \label{eq:1}
\begin{aligned}
(-\sum_{j\in\mathcal{N}_i}(\bm{\lambda}_{i,j}^{k+1}+\rho(\bm{x}_j^{k+1}-\bm{x}_j^k))-\gamma_i\rho(\bm{x}_i^{k+1}-\bm{x}_i^k)
+A_i^T\bm{\lambda}^{k+1*})^T(\bm{x}_i^{k+1}-\bm{x}_i^*)\ge m_f\parallel\bm{x}_i^{k+1}-\bm{x}_i^*\parallel^2.
\end{aligned}
\end{equation}

Based on the definition of $A$, $D$, and $U$, one can further have
\begin{eqnarray}
&&\sum\limits_{j\in\mathcal{N}_i}\bm{\lambda}_{i,j}^{k+1}=A_i^T\bm{\lambda}^{k+1},\\
&&\sum\limits_{j\in\mathcal{N}_i}(\bm{x}_j^{k+1}-\bm{x}_j^k)=(D-A^TA)_i^T(\bm{x}^{k+1}-\bm{x}^k),\\
&&\gamma_i\rho(\bm{x}_i^{k+1}-\bm{x}_i^k)=\rho U_i(\bm{x}^{k+1}-\bm{x}^k). 
\end{eqnarray}

Recall that $Q=D-A^TA+U$ and $Q=Q^T$, we can combine \eqref{eq:1} with the above three equations to obtain
\begin{equation}
\begin{aligned} \label{eq:2}
(-A_i^T(\bm{\lambda}^{k+1}-\bm{\lambda}^{k+1*})-\rho Q_i^T(\bm{x}^{k+1}-\bm{x}^k))^T
\bm{\cdot}(\bm{x}_i^{k+1}-\bm{x}_i^*)\ge m_f\parallel\bm{x}_i^{k+1}-\bm{x}_i^*\parallel^2.
\end{aligned}
\end{equation}

Summing both sides of \eqref{eq:2} over $i=1,2,\ldots,2N$ and using
\begin{equation}
\begin{aligned}	
\sum\limits_{i=1}^{2N}(\bm{x}_i^{k+1}-\bm{x}_i^{*})^T(A_i)^T(\bm{\lambda}^{k+1}-\bm{\lambda}^{k+1*})
=(\bm{x}^{k+1}-\bm{x}^{*})^TA^T(\bm{\lambda}^{k+1}-\bm{\lambda}^{k+1*}) \nonumber\\	
\sum\limits_{i=1}^{2N}(\bm{x}_i^{k+1}-\bm{x}_i^{*})^T\rho Q_i^T(\bm{x}^{k+1}-\bm{x}^k)
=\rho(\bm{x}^{k+1}-\bm{x}^{*})^TQ^T(\bm{x}^{k+1}-\bm{x}^k)
\end{aligned}
\end{equation}
we have
\begin{equation} \label{eq:3}
\begin{aligned}
-(\bm{x}^{k+1}-\bm{x}^{*})^TA^T(\bm{\lambda}^{k+1}-\bm{\lambda}^{k+1*})
-\rho(\bm{x}^{k+1}-\bm{x}^{*})^TQ^T(\bm{x}^{k+1}-\bm{x}^k)\ge m_f\parallel\bm{x}^{k+1}-\bm{x}^*\parallel^2
\end{aligned}
\end{equation}

Moreover,  the following equalities can be obtained by using algebraic manipulations:
\begin{equation}
\begin{aligned}
(\bm{x}^{k+1}-\bm{x}^{*})^TQ^T(\bm{x}^{k+1}-\bm{x}^{k})=\frac{1}{2}\parallel \bm{x}^{k+1}-\bm{x}^{k}\parallel_Q^2
+\frac{1}{2}(\parallel \bm{x}^{k+1}-\bm{x}^*\parallel_{Q}^2-\parallel \bm{x}^{k}-\bm{x}^*\parallel_{Q}^2),  \label{eq:4} 
\end{aligned}
\end{equation}
\begin{equation}
\begin{aligned}
(\bm{x}^{k+1}-\bm{x}^{*})^TA^T(\bm{\lambda}^{k+1}-\bm{\lambda}^{k+1*})
=\frac{1}{\rho}(\bm{\lambda}^{k+1}-\bm{\lambda}^k)^T(\bm{\lambda}^{k+1}-\bm{\lambda}^{k+1*})\\
\end{aligned}
\end{equation}
\begin{equation}
\begin{aligned}
\frac{1}{\rho}(\bm{\lambda}^{k+1}-\bm{\lambda}^k)^T(\bm{\lambda}^{k+1}-\bm{\lambda}^{k+1*})=\frac{1}{2\rho}\parallel \bm{\lambda}^{k+1}-\bm{\lambda}^{k}\parallel^2
-\frac{1}{2\rho}\parallel \bm{\lambda}^k-\bm{\lambda}^{k+1*}\parallel^2+ \frac{1}{2\rho}\parallel \bm{\lambda}^{k+1}-\bm{\lambda}^{k+1*}\parallel^2
\end{aligned}
\end{equation}

Based on the above three inequities, \eqref{eq:3} can be rewritten as
\begin{equation} \label{eq:5}
\begin{aligned}
m_f\parallel\bm{x}^{k+1}-\bm{x}^*\parallel^2\le&-\frac{\rho}{2}\parallel \bm{x}^{k+1}-\bm{x}^*\parallel_{Q}^2-
\frac{1}{2\rho}\parallel \bm{\lambda}^{k+1}-\bm{\lambda}^{k+1*}\parallel^2+\frac{\rho}{2}\parallel \bm{x}^{k}-\bm{x}^*\parallel_{Q}^2+\frac{1}{2\rho}\parallel \bm{\lambda}^k-\bm{\lambda}^{k+1*}\parallel^2 \\
&-\frac{\rho}{2}\parallel \bm{x}^{k+1}-\bm{x}^{k}\parallel_Q^2-\frac{1}{2\rho}\parallel \bm{\lambda}^{k+1}-\bm{\lambda}^{k}\parallel^2
\end{aligned}
\end{equation}

Recall that  $H={\rm diag}\{\rho Q, \frac{1}{\rho}I_{|E|n}\}$ and $\bm{y}^k=[\bm{x}^{kT},\bm{\lambda}^{kT}]^T$. The above inequality can be simplified as
\begin{equation} \label{eq:mf}
\begin{aligned}
2m_f\parallel\bm{x}^{k+1}-\bm{x}^*\parallel^2\le\parallel \bm{y}^{k}-\bm{y}^{k+1*}\parallel_{H}^2 
-\parallel \bm{y}^{k+1}-\bm{y}^{k+1*}\parallel_{H}^2
-\parallel \bm{y}^{k+1}-\bm{y}^{k}\parallel_H^2
\end{aligned}
\end{equation}       

On the other hand,  note that for any constant $u>1$,  the following relationship is true \cite{ling2014decentralized}
\begin{equation} \label{eq:6}
\begin{aligned}
(u-1)\parallel \bm{a}-\bm{b}\parallel^2\ge(1-\frac{1}{u})\parallel\bm{ b}\parallel^2-\parallel \bm{a}\parallel^2
\end{aligned}
\end{equation}

So we have
\begin{equation} \label{eq:h derivative}
\begin{aligned}
(u-1)\parallel \triangledown h^{k+1}(\bm{x}^{k+1})-\triangledown h^{k+1}(\bm{x}^*)\parallel^2 
&=(u-1)\parallel A^T(\bm{\lambda}^{k+1}-\bm{\lambda}^{k+1*})+\rho Q^T(\bm{x}^{k+1}-\bm{x}^k)\parallel^2 \\
&\ge\frac{u-1}{u}\parallel A^T(\bm{\lambda}^{k+1}-\bm{\lambda}^{k+1*})\parallel^2 -\parallel \rho Q^T(\bm{x}^{k+1}-\bm{x}^k)\parallel^2 
\end{aligned}
\end{equation}

Since $\bm{\lambda}^{k+1}$ and $\bm{\lambda}^{k+1*}$ lie in the column space of of $A$, we have \cite{ling2014decentralized}
\begin{equation} \label{eq:Amin}
\begin{aligned}
&\parallel A^T(\bm{\lambda}^{k+1}-\bm{\lambda}^{k+1*})\parallel^2\ge A_{\min} \parallel \bm{\lambda}^{k+1}-\bm{\lambda}^{k+1*}\parallel^2\\
&\parallel \rho Q^T(\bm{x}^{k+1}-\bm{x}^k)\parallel^2 \le \rho^2 Q_{\max}\parallel\bm{x}^{k+1}-\bm{x}^k\parallel_Q^2
\end{aligned}
\end{equation}
where $Q_{\max}$ is the largest eigenvalue of $Q$, $A_{\min}$ is the smallest nonzero eigenvalue of $A^TA$.

In addition, given that $\parallel \triangledown h^{k+1}(\bm{x}^{k+1})-\triangledown h^{k+1}(\bm{x}^*)\parallel^2\le L^2\parallel  \bm{x}^{k+1}- \bm{x}^*\parallel^2$ is true according to Assumption \ref{assumption gi}, using \eqref{eq:Amin} and \eqref{eq:h derivative}, we can obtain
\begin{equation} \label{eq: L}
\begin{aligned}
(u-1)L^2\parallel  \bm{x}^{k+1}- \bm{x}^*\parallel^2\ge\frac{(u-1)A_{\min}}{u}\parallel \bm{\lambda}^{k+1}-\bm{\lambda}^{k+1*}\parallel^2 
-\rho^2 Q_{\max}\parallel\bm{x}^{k+1}-\bm{x}^k\parallel_Q^2
\end{aligned}
\end{equation}

Using algebraic manipulations, the above inequality can be rewritten as
\begin{equation} \label{eq: L1}
\begin{aligned}
\frac{u Q_{\max}}{(u-1)A_{\min}}\rho \parallel\bm{x}^{k+1}-\bm{x}^k\parallel_Q^2+\frac{uL^2}{\rho A_{\min}}\parallel  \bm{x}^{k+1}- \bm{x}^*\parallel^2 
\ge\frac{1}{\rho}\parallel \bm{\lambda}^{k+1}-\bm{\lambda}^{k+1*}\parallel^2
\end{aligned}
\end{equation}

Adding $\frac{u Q_{\max}}{(u-1)A_{\min}}\frac{1}{\rho}\parallel \bm{\lambda}^{k+1}-\bm{\lambda}^{k}\parallel^2$ and $\rho Q_{\max} \parallel\bm{x}^{k+1}-\bm{x}^*\parallel^2$ to the left hand side of the above inequality, and adding $\rho \parallel\bm{x}^{k+1}-\bm{x}^*\parallel_Q^2$ to the right hand side, we obtain the following inequality based on the fact $\rho \parallel\bm{x}^{k+1}-\bm{x}^*\parallel_Q^2\le\rho Q_{\max} \parallel\bm{x}^{k+1}-\bm{x}^*\parallel^2$:
\begin{equation} \label{eq: L2}
\begin{aligned}
&\frac{u Q_{\max}}{(u-1)A_{\min}}(\rho \parallel\bm{x}^{k+1}-\bm{x}^k\parallel_Q^2+\frac{1}{\rho}\parallel \bm{\lambda}^{k+1}-\bm{\lambda}^{k}\parallel^2)
+(\frac{uL^2}{\rho A_{\min}}+\rho Q_{\max})\parallel  \bm{x}^{k+1}- \bm{x}^*\parallel^2 \\
&\ge\frac{1}{\rho}\parallel \bm{\lambda}^{k+1}-\bm{\lambda}^{k+1*}\parallel^2+\rho \parallel\bm{x}^{k+1}-\bm{x}^*\parallel_Q^2
\end{aligned}
\end{equation}

Let
\begin{equation}
\begin{aligned}
\delta=\min\{\frac{(u-1)A_{\min}}{u Q_{\max}}, \frac{2m_fA_{\min}\rho}{uL^2+\rho^2A_{\min}Q_{\max}}\}
\end{aligned}
\end{equation}

Inequality \eqref{eq: L2} becomes
\begin{equation} \label{eq: LL}
\begin{aligned}
\frac{1}{\delta} \parallel\bm{y}^{k+1}-\bm{y}^k\parallel_H^2
+\frac{2m_f}{\delta}\parallel  \bm{x}^{k+1}- \bm{x}^*\parallel^2 
\ge\parallel \bm{y}^{k+1}-\bm{y}^{k+1*}\parallel_H^2
\end{aligned}
\end{equation} 

Based on \eqref{eq:mf} and \eqref{eq: LL}, we can get
\begin{equation} \label{eq: final}
\begin{aligned}
\frac{1}{\delta} \parallel\bm{y}^{k}-\bm{y}^{k+1*}\parallel_H^2
-\frac{1}{\delta}\parallel  \bm{y}^{k+1}- \bm{y}^{k+1*}\parallel_H^2 
\ge\parallel \bm{y}^{k+1}-\bm{y}^{k+1*}\parallel_H^2
\end{aligned}
\end{equation} 
which proves Lemma \ref{lemma 1}.

\subsection{Proof of Lemma \ref{lemma 2}}
First, we have
\begin{equation}\label{eq:a}
\begin{aligned}
\parallel \bm{y}^{k}-\bm{y}^{k+1*}\parallel_H-	\parallel \bm{y}^{k}-\bm{y}^{k*}\parallel_H 
\le\parallel \bm{y}^{k+1*}-\bm{y}^{k*}\parallel_H
\end{aligned}
\end{equation}

On the other hand, we have   \begin{equation} 
\begin{aligned}
\parallel \bm{y}^{k+1*}-\bm{y}^{k*}\parallel_H =\frac{1}{\sqrt{\rho}}\parallel \bm{\lambda}^{k+1*}-\bm{\lambda}^{k*}\parallel
\end{aligned}
\end{equation} 
\begin{equation}
\begin{aligned}
\parallel A^T(\bm{\lambda}^{k+1*}-\bm{\lambda}^{k*})\parallel=\parallel\triangledown h^{k+1}(\bm{x}^*)-\triangledown h^k(\bm{x}^*) \parallel.
\end{aligned}
\end{equation} 

Therefore, we can get the following inequality using \eqref{eq:Amin}
\begin{equation} \label{eq: b}
\begin{aligned}
\parallel \bm{\lambda}^{k+1*}-\bm{\lambda}^{k*}\parallel\le\frac{1}{\sqrt{A_{\min}}}\parallel\triangledown h^{k+1}(\bm{x}^*)-\triangledown h^k(\bm{x}^*) \parallel.
\end{aligned}
\end{equation} 

Combing \eqref{eq:a} to  \eqref{eq: b}, we obtain
\begin{equation} 
\begin{aligned}
\parallel \bm{y}^{k}-\bm{y}^{k+1*}\parallel_H
\le\parallel \bm{y}^{k}-\bm{y}^{k*}\parallel_H+\frac{1}{\sqrt{\rho A_{\min}}}\parallel\triangledown h^{k+1}(\bm{x}^*)-\triangledown h^k(\bm{x}^*) \parallel,
\end{aligned}
\end{equation}
which completes the proof of Lemma \ref{lemma 2}.

\end{document}